\documentclass[11pt,a4paper]{amsart}

\usepackage{latexsym}
\usepackage{verbatim}

\usepackage{amssymb,amsmath,amscd,graphicx,color,enumerate}

\usepackage{hyperref}

 \usepackage[normalem]{ulem}
 
\renewcommand\sout{\bgroup\markoverwith
 {\textcolor{red}{\rule[0.7ex]{3pt}{1.4pt}}}\ULon}

 \usepackage{url}
 \usepackage[top=1in, bottom=1.5in, left=1.0in, right=1.0in]{geometry}
 \usepackage{booktabs}
\usepackage[font=bf]{caption}

\usepackage{microtype}

\usepackage[all]{xy}

 \usepackage{color}                

\definecolor{darkblue}{rgb}{0.2, 0., .7}
\newcommand{\End}{\operatorname{End}}

\newcommand{\Ind}{\operatorname{Ind}}

\newcommand{\C}{\mathbb C}

\newcommand{\R}{\mathbb R}

\newcommand{\Z}{\mathbb Z}

\newcommand{\maC}{\mathcal C}
\newcommand{\maD}{\mathcal D}

\newcommand{\maH}{\mathcal H}
\newcommand{\maK}{\mathcal K}
\newcommand{\maL}{\mathcal L}

\newcommand{\maP}{\mathcal P}
\newcommand{\maQ}{\mathcal Q}

\newcommand{\oid}{\operatorname{Id}}

\newtheorem{theorem}{Theorem}[section]
\newtheorem{lemma}[theorem]{Lemma}
\newtheorem{proposition}[theorem]{Proposition}
\newtheorem{corollary}[theorem]{Corollary}

\theoremstyle{definition}
\newtheorem{definition}[theorem]{Definition}
\newtheorem{remark}[theorem]{Remark}

\newtheorem*{proposition*}{Proposition}
\newtheorem*{definition*}{Definition}
\newtheorem*{theorem*}{Theorem}

\title[Fredholm Criteria for $G$-pseudodifferential Operators]{Fredholm Criteria for  $G$-pseudodifferential Operators
}
  
\author[A. Baldare]{Alexandre Baldare}
\address{Alexandre Baldare. Leibniz University Hannover, Institute of Analysis, Welfengarten 1, 30167 Hannover, Germany}
\email{alexandre.baldare@ac-montpellier.fr}
\urladdr{https://baldare.github.io/Baldare.Alexandre/}

\author[A.Yu. Savin]{Anton Yu. Savin}
 \address{Anton Yu. Savin.  
Peoples' Friendship University of Russia (RUDN University),
6 Miklukho-Maklaya St
117198 Moscow, Russia
}
\email{antonsavin@mail.ru}
\urladdr{https://eng.rudn.ru/science/rudn-scientists/anton-yu-savin/}

\author[E. Schrohe]{Elmar Schrohe}
 \address{Elmar Schrohe. 
 Leibniz University Hannover, Institute of Analysis, Welfengarten 1, 30167 Hannover, Germany}
\email{schrohe@math.uni-hannover.de}
\urladdr{https://www.analysis.uni-hannover.de/de/schrohe/}

  \thanks{
   {\em \hspace*{1em} AMS Subject classification:} 
   47A53, 
   58J40, 
    47L80, 
    46N20. 
     }

\begin{document}

\begin{abstract}
Let $G$ be a compact Lie group that acts smoothly on a closed manifold $M$. 
Using a general Simonenko principle, we derive a novel criterion for the Fredholm property of $G$-pseudodifferential operators acting on Sobolev spaces of sections of vector bundles over $M$. 
In case the group is finite, we obtain  a further characterization of the Fredholm property of
$G$-pseudodifferential operators in terms of the invertibility of suitable symbols. 
\end{abstract}

\maketitle \tableofcontents

\section*{Introduction}

This paper is devoted to the investigation of Fredholm criteria for a certain class of operators on closed manifolds. More precisely, let $G$ be a compact Lie group which acts smoothly on a compact manifold $M$ without boundary. There is no loss of generality in assuming that $M$ is Riemannian and that $G$ acts on $M$ by isometries. Let $E^\pm \rightarrow M$ be $G$-equivariant vector bundles equipped with $G$-invariant Hermitian metrics. We are interested in so-called $G$-pseudodifferential operators, i.e., operators of the form
\begin{equation}\label{eq.intro.Gop}
P + \maD : H^{s}(M,E^+) \rightarrow H^{s-m}(M,E^-),\quad (P + \maD) u=Pu + \int_G D_g (T_gu) \ dg,
\end{equation}
where $P$ is a pseudodifferential operator of order $m$, $D_g$, $g\in G$, is a smooth family of pseudodifferential operators of order $m$, $T_g$, $g\in G$,  is a family of so-called shift operator defined by
$T_gu(x)=g(u(g^{-1}x))$, and $dg$ is the normalized Haar measure of $G$; $H^s(M,E^+)$ and $H^{s-m}(M,E^-)$ are the standard Sobolev spaces of sections of $E^\pm$. 

When the operator $P$ is elliptic, it is in particular Fredholm.  Therefore, if $P$ is elliptic, using a parametrix for the operator $P$, we see that we can assume that $E^+=E^-=E$, $P=\oid$, and that $D_g$ is a smooth family of pseudodifferential operators of order $0$ acting on the Hilbert space $L^2(M,E)$ of square integrable sections of $E$, see Section~\ref{sec.preliminaries}. 
For a general reference on the class of $G$-pseudodifferential operators we refer to 
\cite{AL94,
ABL98b,Savin.unif,Sternin.unif} and the references therein.  

\medskip
In Section \ref{sec.Simonenko}, we use the main theorem of \cite{Baldare.Simonenko} in order to derive a Fredholm criterion for $G$-pseudodifferential operators. More precisely, we use a local principle  \`a la Simonenko \cite{simonenko1965new}, see also \cite{allan,douglas} for more general local principles. We would like to point out that local principles were intensively 
used to obtain Fredholm conditions for many classes of singular operators, see for examples 
\cite{ABL98b,bottcher,Kisil_2012,PlSe2,NaSaSt4,Schulze.localPrinciple,semenyuta
} and the references therein. In particular in \cite{Baldare.Simonenko}, the first author of this paper 
showed a general Simonenko principle using $C^*$-algebra techniques and then derived from this result a Fredholm criterion   for the restriction of $G$-invariant pseudodifferential operators to isotypical components. In this paper, we apply the same method to characterize Fredholm $G$-pseudodifferential operators, and we obtain the following result; see Section~\ref{sec.Simonenko} for more details.

\begin{theorem*}[Simonenko's principle for $G$-operators] 
\em Assume that $M$ does not contain any clopen orbit. 
Let $\maD$ be a $G$-pseudodifferential operator as in Equation \eqref{eq.intro.Gop}. Then
\begin{enumerate}\renewcommand\labelenumi{\rm(\alph{enumi})}
\item  $\oid + \maD$ commutes modulo compact operators with $\maC(M)^G=\maC(M/G)$.
\item $\oid + \maD: L^2(M,E) \to L^2(M,E)$ is Fredholm if, and only if, it  is locally invertible on $M/G$.
\end{enumerate}
\end{theorem*}

See Section \ref{sec.Simonenko} for the precise definition of locally invertible
operators. 
\medskip

In Section~\ref{sec.finite} we consider the case, where the group is finite. To avoid any confusion, 
we shall then  denote our group by $\Gamma$.  So  $\Gamma$ is a finite group acting by isometries on $M$ and $E$, and we consider $\Gamma$-pseudodifferential operators, i.e., operators of the form 
\begin{eqnarray}
\label{operatorsD}
\maD = \sum_{\gamma\in \Gamma}D_\gamma T_\gamma: L^2(M,E) \to L^2(M,E),
\end{eqnarray}
where the $\maD_\gamma$ are pseudodifferential operators of order zero acting on sections of $E$, and the $T_\gamma$ are shift operators as before. 
As $\Gamma$ is finite, the identity operator is already included in this class; this explains the simpler form compared to \eqref{eq.intro.Gop}.  

 We start by recalling the pseudodifferential uniformization principle shown in \cite{Sternin.unif}, see also \cite{SaScSt15} or \cite{Savin.unif} in our simpler case of finite groups.
In particular, this allows to view a $\Gamma$-pseudodifferential operator 
$\maD : L^2(M,E) \rightarrow L^2(M,E)$ as the restriction to $\Gamma$-invariant sections of a $\Gamma$-invariant pseudodifferential operator $\hat{\maD} : L^2(M,E \otimes \C[\Gamma]) \rightarrow L^2(M,E \otimes \C[\Gamma])$ with coefficients in the Hermitian vector bundle $E \otimes \C[\Gamma]$. We then use the main result of \cite{BCLN2} to characterize the Fredholm property of the $\Gamma$-invariant pseudodifferential operator $\hat{\maD}$ restricted to the subspaces of invariant sections in terms of its principal symbol. 

In order to simplify the statement of our result, we first assume that $M/\Gamma$ is connected 
and let $\Gamma_0$ be a minimal isotropy subgroup of $\Gamma$, see \cite{Bredon,tomDieckTransBook}. 
We obtain the following theorem.

\begin{theorem*}\em
Let $\maD$ be a $\Gamma$-operator as in \eqref{operatorsD}. 
Then the following assertions are equivalent.
\begin{enumerate}\renewcommand\labelenumi{\rm(\roman{enumi})} 
\item The operator $ \maD:L^2(M,E)\to L^2(M,E)$ is Fredholm.
\item The restriction to $L^2(M,E \otimes \C[\Gamma])^{\Gamma}$ 
of the pseudodifferential $\Gamma$-invariant operator $\hat{\maD}$ introduced in Theorem~\ref{thm.unif} is Fredholm.
\item The restriction of the principal symbol
\begin{equation*}
\sigma_0(\hat{\maD})(\xi) : \big(E_\xi \otimes \C[\Gamma]\big)^{\Gamma_0} \rightarrow \big(E_\xi \otimes \C[\Gamma]\big)^{\Gamma_0} 
\end{equation*}
is invertible for all $ \xi \in S^*M^{\Gamma_0}$.
\end{enumerate} 
\end{theorem*}

We finally explain how to relax the hypothesis that $M/\Gamma$ is connected, see Proposition \ref{prop.reduction}, and discuss some standard cases:  trivial actions, topologically free actions and some non-topologically free actions. 

In  \cite[Theorem 4]{Antonevich1973} Antonevich has shown that the invertibility of the principal symbol 
\begin{equation*}
\sigma_0(\hat{\maD})(\xi) : E_\xi \otimes \C[\Gamma] \rightarrow E_\xi \otimes \C[\Gamma]  
\end{equation*}
is sufficient for the Fredholm property of $\maD$ and that it is also necessary,  if the action of $\Gamma$ is efficient. If the action is efficient, a minimal isotropy group is trivial, and we recover Antonevich's result. On the other hand, if a minimal isotropy group is not trivial, we  obtain a new necessary and sufficient condition for the Fredholm property.

\medskip

\noindent
{\em{Acknowledgements.}} The authors would like to thank RFBR,  project number 21-51-12006, and DFG, project SCHR 319/10-1 as well as the DFG Priority Programme 2026 `Geometry at Infinity'   
    for their support.
 The first named author would like to express his gratitude to V. Nistor for useful discussions, suggestions and encouragement during the redaction of this paper.   
 He also wishes to dedicate this work to the memory of M. Benameur, his doctoral advisor, who passed away in September 2025 and who taught him to reach for the stars -- and $C^*$-algebras were the closest ones.

\section{Preliminaries}\label{sec.preliminaries}

Let $G$ be a compact Lie group which acts smoothly on a compact Riemannian manifold $M$ without boundary. 
Without loss of generality, we assume that $M$ is endowed with a $G$-invariant metric and that $G$ acts by isometries.
Let $E^\pm \rightarrow M$ be $G$-equivariant vector bundles equipped with $G$-invariant Hermitian metrics.
We shall denote by $\psi^m(M,E^+,E^-)$ the space of  
classical pseudodifferential operators of order $m$, acting between
sections of the $G$-equivariant vector bundles $E^\pm \rightarrow M$.
As pointed out above, we are interested in operators of the form
\begin{equation*}
P + \maD : C^\infty(M,E^+) \rightarrow C^\infty(M,E^-),\qquad (P + \maD) s(x)=Ps(x) + \int_G D_g (T_gs) (x) dg,
\end{equation*}
where $P \in \psi^m(M,E^+,E^-)$ is elliptic, $D_g \in \psi^m(M,E^+,E^-)$, $g\in G$, is a smooth family of pseudodifferential operators, $T_g$, $g\in G$,  are given by 
$T_gs(x)=g(s(g^{-1}x))$ and $dg$ is the normalized Haar measure of $G$.

Let us introduce the class $\psi^m_G(M,E^+,E^-)$ of $G$-pseudodifferential operators, 
that is operators of the form $\maD = \int_G D_g T_g\ dg$ as before.
The class $\psi^m_G(M;E^+,E^-)$ coincides with 
the class of operators of the form
$\int_G T_g D_g\ dg$ because $\int_G T_gD_g\ dg = \int_G T_g D_g T_{g^{-1}} T_g\ dg$ and 
$T_g D_g T_{g^{-1}} \in \psi^m(M,E^+,E^-)$ is a smooth family 
(with respect to $g\in G$) of pseudodifferential operators,
see \cite{AS1,Sternin.unif}.

Let $\nabla^\pm$ be  metric preserving, $G$-invariant connections on $E^\pm$. 
In the sequel, $\Delta^\pm := \nabla^{\pm *} \nabla^\pm$ will
denote the (positive) Laplacian on $M$ with coefficients
in $E^\pm$.  Since $M$ is compact, the Sobolev spaces $H^s(M, E^\pm)$ can be defined for any
$s \in [0,\infty)$ as the domain of $(\oid + \Delta^\pm)^{\frac{s}{2}}$ in $L^2(M, E^\pm)$, so that 
$(\oid + \Delta^\pm)^{\frac{s}{2}} : H^s(M, E^{\pm}) \rightarrow L^2(M, E^\pm)$ is
an isomorphism. For $s < 0$, we define $H^s(M, E^\pm)$ as the dual of $H^{-s}(M, E^\pm)$ with
respect to the pairing defined by the $L^2$-inner product. 
In order to simplify the notation, we shall denote $\Delta^\pm$ 
by $\Delta$ in the sequel 
since the notation should be clear from the context.

Recall that a  classical pseudodifferential operator
$P \in \psi^m(M;E^+,E^-)$ extends to a bounded linear map 
\begin{equation*}
P : H^{s}(M,E^+) \rightarrow H^{s-m}(M,E^-).
\end{equation*}
Furthermore, if $P \in \psi^m(M;E^+,E^-)$ and $\maD \in \psi^m_G(M,E^+,E^-)$ 
then $P+\maD$ extends to a bounded operator
\begin{equation}
P+\maD : H^s(M,E^+) \rightarrow H^{s-m}(M,E^-),
\end{equation}
with norm controlled by 
\begin{equation*}
\|P+\maD\| \leq \|P\| + \int_G \|D_g T_g\| dg \le \|P\| + \int_G \|D_g\| dg.
\end{equation*}


We shall denote by $\maL(H^s(M, E^+), H^{s-m}(M, E^-))$ the set of bounded linear maps 
from $ H^s(M, E^+)$ to $H^{s-m}(M, E^-)$. The closures of
$\psi^0(M; E^+,E^-)$ and $\psi^0_G(M; E^+,E^-)$, respectively, in the norm topology of
$\maL(L^2(M, E^+), L^2(M, E^-))$ 
will be denoted in the sequel by  
$\Psi(M;E^+,E^-)$ and  $\Psi_G(M;E^+,E^-)$, respectively.

When $E^+=E^-=E$, we shall use the standard conventions $\maL(L^2(M,E),L^2(M,E)) = \maL(L^2(M,E))$,
$\psi^m(M;E,E)=\psi^m(M,E)$, $\psi^m_G(M;E,E)=\psi^m_G(M,E)$, 
$\Psi(M;E,E)=\Psi(M,E)$ and $\Psi_G(M;E,E)=\Psi_G(M,E)$.
The sets $\Psi(M,E)$ and $\Psi_G(M,E)$ 
are sub-$C^*$-algebras of $\maL(L^2(M,E))$. 
Notice that in general $\Psi_G(M,E)$ is non-unital, i.e., in general $\oid \not \in \Psi_G(M,E)$.

We rely on Atkinson's theorem to characterize Fredholm operators: 
A bounded linear operator is Fredholm if, and only if, it is invertible modulo compact operators.
We recall the following standard lemma concerning order reduction for $G$-operators, see e.g. \cite{BCLN1}.
\begin{lemma}
\begin{enumerate}\renewcommand\labelenumi{\rm (\alph{enumi})}
\item A bounded operator $A : H^s(M, E^+) \rightarrow H^{s-m}(M, E^-)$ is Fredholm if, and only if, 
$\tilde{A} := (\oid + \Delta)^{\frac{s-m}{2}} A (\oid + \Delta)^{-\frac{s}{2}} 
				: L^2(M, E^+) \rightarrow L^2(M, E^-)$ 
is Fredholm.
\item Let $A \in \maL(H^s(M,E^+),H^{s-m}(M,E^-))$ be a Fredholm operator, $S$ an inverse modulo compact operators for $A$, and  $B \in \maL(H^s(M,E^+), H^{s-m}(M,E^-))$.
Then
 $A+B$ is Fredholm if, and only if, $\oid + SB$ is Fredholm.  
\item Let  $E=E^+ \oplus E^-$. Then 
$P \in \maL(L^2(M,E^+),L^2(M,E^-))$ is Fredholm if, and only if, 
$\begin{pmatrix}
0& P^*\\
P&0
\end{pmatrix}\in \maL(L^2(M,E))$ is Fredholm.
\item  If $P$ is the limit in the norm
topology of $\maL(H^s(M, E^+), H^{s-m}(M, E^-))$ of a sequence of operators $P_n \in \psi^m(M, E^+,E^-)$, 
then $\tilde {P}$, defined as in (a), belongs to $\Psi(M,E^+,E^-):=\overline{\psi^0(M; E^+,E^-)}$.
\item If $\maD$ is the limit in the norm
topology of $\maL(H^s(M, E^+), H^{s-m}(M, E^-))$ of a sequence of operators $\maD_n \in \psi^m_G(M, E^+,E^-)$, 
then $\tilde{\maD}$, defined as in (a) belongs to $\Psi_G(M,E^+,E^-):=\overline{\psi^0_G(M; E^+,E^-)}$.
\item If $R$ is the limit in the norm
topology of $\maL(H^s(M, E^+), H^{s-m}(M, E^-))$ of a sequence of operators $P_n+\maD_n \in \psi^m(M;E^+,E^-)+\psi^m_G(M, E^+,E^-)$, an $\tilde R$ is as in (a), 
then $\tilde{R} \in \Psi_G^{\mathrm{full}}(M,E^+,E^-):=\overline{\psi^0(M;E^+,E-)+\psi^0_G(M; E^+,E^-)}$.
\end{enumerate}
\end{lemma} 

\begin{proof}
(a) The operator $P$ is Fredholm if, and only if, $\tilde{P}$ is Fredholm, because 
 $(\oid +\Delta)^{s} : H^m(M,E^\pm) \rightarrow H^{m-2s}(M,E^\pm)$ is an isomorphism for all $m \in \R$.
 
(b) This is a trivial consequence of Atkinson's theorem. 

(c) This is clear. 

(d)  According to \cite{Seeley.Comp.Power.ellip}, the complex powers of the Laplace operator are classical pseudodifferential operators. Hence the result is a consequence of the
continuity of the multiplication in the operator norm.

(e) Notice that if $\maD \in \psi^m_G(M;E^+,E^-)$ then $\tilde{\maD}\in \psi^0_G(M;E^+,E^-)$.
 Indeed, the $G$-invariance of the Laplace operator implies that 
$(\oid + \Delta)^{\frac{s-m}{2}}D_g T_g (\oid + \Delta)^{-\frac{s}{2}} = 
(\oid + \Delta)^{\frac{s-m}{2}}D_g  (\oid + \Delta)^{-\frac{s}{2}}T_g$ 
and it follows again from Seeley's result \cite{Seeley.Comp.Power.ellip} that 
$(\oid + \Delta)^{\frac{s-m}{2}}D_g  (\oid + \Delta)^{-\frac{s}{2}}$
is an order $0$ classical pseudodifferential operator.
Now assume that $\maD$ is the limit in the norm
topology of $\maL(H^s(M, E^+), H^{s-m}(M, E))$ of a sequence of operators $\maD_n \in \psi^m_G(M, E^+,E^-)$.
Then we get from the previous discussion that $\tilde{\maD}_n \in \psi^0_G(M;E^+,E^-)$.
The result follows then as above from the continuity of the multiplication
in the operator norm.

(f) The last item is completely similar.
\end{proof}

\section{A Simonenko local principle for $G$-pseudodifferential operators}
\label{sec.Simonenko}

We start by recalling the definition of locally invertible operators and the general Simonenko principle shown in \cite{Baldare.Simonenko}, where more references can be found.

Let $\maH$ be a Hilbert space, $T$ a compact Hausdorff space  and  $\maC(T) \hookrightarrow \maL(\maH)$  a non-degenerate faithful
representation (i.e. $\maC(T)$ identifies with its image in $\maL(\maH)$ and the image of the
constant function $1$ is $\oid \in \maL(\maH)$). 
Assume that the image of $\maC(T)$ does not intersect
 $\maK(\maH) \setminus \{0\}$. In other words, we are assuming that $\maC(T)$ is a unital sub-$C^*$-algebra
of the Calkin algebra $\maQ(\maH) := \maL(\maH)/\maK(\maH)$. We shall denote by $M_f$ the image of a
function $f\in \maC(T)$ in $\maL(\maH)$ and call it the multiplication operator by $f$.

Define $\Psi_{T}(\maH) \subset \maL(\maH)$ as the $C^*$-algebra consisting of all
$P \in \maL(\maH)$ such that $M_\phi PM_\psi  \in \maK(\maH)$ is a compact operator 
for all $\phi,\psi \in \maC(T)$ with disjoint support.  
By  Kasparov's lemma (see~\cite[Lemma 5.4.7]{HiRo1})  $\Psi_T(\maH)$ consists of all operators $P$ such that $PM_f - M_f P \in \maK(\maH)$ for all $ f \in \maC(T)$ (pseudolocal operators).


\begin{definition}[see \cite{Baldare.Simonenko,BCLN2}]\label{locinv}
An operator $P \in \maL(\maH)$ is  {\em locally invertible at $x \in T$}  if
there exist:
\begin{enumerate}
\item[(i)] a neighbourhood $V_x$ of $x$ and
\item[(ii)] operators $Q^x_1$ and $Q^x_2 \in \maL(\maH)$
\end{enumerate}
such that, for all $f \in \maC_c(V_x)$
\begin{equation*}
Q^x_1 PM_f = M_f =M_fPQ^x_2 \in \maL (\maH).
\end{equation*}
The operator $P$ is  {\em locally invertible} if it is locally invertible at any $x \in T$.
\end{definition} 

We shall say that $\maC(T) \hookrightarrow \maL(\maH)$ has the property of {\em strong convergence to $0$} if  for all
$x\in T$, all $ h\in \maH$, and all $ \varepsilon >0$, there exists a neighborhood $V$ of $x$
such 
$\|M_{\chi}h\|<\varepsilon$
for every  $\chi \in \maC(T,[0,1])$ that is equal to $1$ on a neighborhood of $x$ and supported in $V$, 
see \cite[Definition 2.9]{Baldare.Simonenko}.

The following was shown in  \cite[Section 2]{Baldare.Simonenko}. 

\begin{proposition}[General Simonenko localization principle]\label{prop.gen.Simonenko} 
Let $P \in \Psi_T(\maH)$.
Assume that $\maC(T)\hookrightarrow \maL(\maH)$ has the property of strong convergence to $0$.
Then
$P$ is locally invertible on $T$ if, and only if, $P$ is Fredholm.
\end{proposition}

Before stating the Simonenko principle for $G$-pseudodifferential operators, 
let us recall more material. Let $\mathfrak{g}$ be the Lie algebra of $G$ and 
recall that any $Y \in \mathfrak{g}$ generates a vector field $Y_M(x) := \frac{d}{dt}_{\vert_{t=0}} e^{tY}\cdot x$.
Let $T^*_GM$ be the closed conical subset of $T^*M$ with fiber in $x\in M$ given by 
\begin{equation}\label{Gtransversal}
(T^*_GM)_x:=\{\xi\in T_x^*M,\ \xi(Y_M(x))=0,\ \forall Y \in \mathfrak{g}\}. 
\end{equation}
We also denote $S^*_GM=T^*_GM\cap S^*M$.

The following lemma was shown in \cite{Baldare.Simonenko}. We recall its easy proof for the benefit of the reader.

\begin{lemma}\label{lemma.finite.component}
\begin{enumerate}\renewcommand\labelenumi{\rm (\alph{enumi})}
\item 
Let $x \in M$ be such that $(T^*_GM)_x=\{0\}$. 
Then the orbit of $x$ is a clopen of $M$. (Here, $M$ needs not to be compact.)
\item The set of points $\maP:=\{x \in M, \ (T^*_GM)_x=\{0\}\}$ is a clopen in $M$.
\end{enumerate}
\end{lemma}
\begin{proof}
(a) Since $(T^*_GM)_x=\{0\}$, we obtain that $S_x=\{x\}$ is the only slice at $x$. From the slice theorem, 
we deduce that the orbit $Gx\cong G\times_{G_x} \{x\}=G \times_{G_x} S_x$ is open but it is also 
compact. Therefore, $Gx$ is a union of connected components of $M$
 in bijection with the connected components of $G/G_x\cong Gx$.
 
(b) It follows from the first part that $\maP$ is a union of connected components of $M$.
Since $M$ is compact, there is only a finite number of connected components.
This implies that $\maP$ is a finite union of clopen and therefore is a clopen. 
\end{proof}

\begin{remark} The set $\maP$ will be empty for example in the following cases:
\begin{enumerate}\renewcommand\labelenumi{\rm (\alph{enumi})}
\item if $M$ is connected and not reduced to a single orbit,
\item if $\dim M > \dim G$.
\end{enumerate}
\end{remark}

Notice that $M$ is the disjoint union of the closed $G$-submanifolds $M\setminus \maP$ and $\maP$. 
Let $\chi_{M\smallsetminus \maP}$ be the characteristic function of $M\smallsetminus \maP$. 
Then the multiplication operator $M_{\chi_{M\smallsetminus \maP}}$ 
by $\chi_{M\smallsetminus \maP}$ is a $G$-invariant orthogonal projection, 
denoted $p$ in the sequel,
with image $L^2(M\smallsetminus \maP , E\vert_{M\smallsetminus \maP})$. 
The corresponding orthogonal decomposition is given by 
$L^2(M,E) =L^2(M\smallsetminus \maP , E\vert_{M\smallsetminus \maP}) \oplus L^2(\maP,E\vert_\maP)$.
Recall that $\maC(M/G)=\maC(M)^G$.

We can now state the main result of this section. Below we consider $L^2(M,E)$ as a $\maC(M/G)$-module.

\begin{theorem}[Simonenko's principle for $G$-operators] \label{thm.Simonenko.Gop}
Let $\maD \in \Psi_{G}(M,E)$ be a $G$-operator as in Section \ref{sec.preliminaries}.
Then $\oid +\maD$ is Fredholm if, and only if, $p \maD p + \oid_{L^2(M\smallsetminus \maP , E\vert_{M\smallsetminus \maP})}$ is locally invertible.
\end{theorem}

\begin{proof}
%

Recall that $p=M_{\chi_{M\smallsetminus \maP}}$, $\oid-p = M_{1-\chi_{M\smallsetminus \maP}}$ and 
let us write 
\begin{equation*}
\maD= p \maD p + (\oid-p) \maD (\oid-p) +p \maD (\oid-p) + (\oid - p) \maD p.
\end{equation*}
The operators $(\oid-p) \maD (\oid-p)$, $p \maD (\oid-p)$ and $(\oid - p) \maD p$ are compact.
Indeed, let 
$\phi,\psi \in \maC(M)^G$ be $G$-invariant functions with disjoint supports. 
Then  $T_g M_\psi=M_\psi T_g$ for all $g\in G$ and therefore
\begin{equation*}
M_\phi \maD M_\psi 
= \int_G M_\phi D_g  M_\psi  T_g \ dg \in \maK(L^2(M,E)), 
\end{equation*}
because $M_{f_1} P M_{f_2}$ is compact whenever $ P \in \Psi(M,E)$ and $f_1,f_2 \in \maC(M)$ have disjoint supports.
This implies that $p \maD (\oid-p)$, and $ (\oid-p) \maD p$ are compact operators.

To show that $(\oid-p)\maD (\oid-p) $ is a compact operator, it is enough to show that 
$(\oid-p) \maD (\oid-p)\in \maK(L^2(\maP,E\vert_\maP))$. 
Recall that $\maP=\sqcup C_i$ is a finite union of clopen orbits $C_i \cong G/G_{x_i}$ with suitable $x_i\in M$ 
and that $E\vert_{C_i} \cong G \times_{G_{x_i}} E_{i}$,
where $E_{i}$ is a unitary representation of $G_{x_i}$.    
Let $\phi_i$ be the characteristic function of $C_i$. Then we can write
\begin{equation*}
M_{1-\chi} \maD M_{1-\chi} = \sum_i M_{\phi_i} \maD M_{\phi_i} + \sum_{i\neq j} M_{\phi_i} \maD M_{\phi_j}.
\end{equation*} 
From the previous discussion we know that $M_{\phi_i} \maD M_{\phi_j}$ is compact if $i \neq j$ 
and therefore $\sum_{i\neq j} M_{\phi_i} \maD M_{\phi_j}$ is compact.
Let us show that also $M_{\phi_i} \maD M_{\phi_i} $ is compact.
Notice that $M_{\phi_i} \maD M_{\phi_i} = \int_G M_{\phi_i} D_g M_{\phi_i} T_g \ dg$ is a $G$-pseudodifferential operator
on $C_i \cong G/G_{x_i}$. By definition, $M_{\phi_i} \maD M_{\phi_i}$ is the limit of sums of operators of the form 
$B \circ T_\varphi$, where $B$  is a bounded operator on $L^2(C_i,E\vert_{C_i})$, 
$T_\varphi:=\int_G \varphi(g) T_g \ dg$ and $\varphi \in C^\infty(G)$.
Since 
$E\vert_{C_i} \cong G \times_{G_{x_i}} E_{i}$ 
and $\varphi \in C^\infty(G)$  the operator 
$ \hat{T}_\varphi:=\int_G \varphi(g) T_g \ dg : L^2(G,E_{i}) \rightarrow L^2(G,E_{i})$ is compact, 
because its Schwartz kernel is given by the smooth function $k_\varphi(g,h):=\varphi(gh^{-1})$. 
It follows that the restriction $\hat{T}_\varphi^{G_{i}}$ of $\hat{T}_\varphi$ to 
$L^2(G,E_{i})^{G_{x_i}}\cong L^2(G/G_{x_i},G \times_{G_{i}} E_{i})\cong L^2(C_i,E\vert_{C_i})$ is compact.
But modulo the isomorphism $L^2(G,E_{i})^{G_{x_i}} \cong L^2(C_i,E\vert_{C_i})$ the operator $\hat{T}_\varphi^{G_{i}}$ is exactly $T_\varphi$,
and this shows that $B \circ T_\varphi$ is compact.
Since the ideal of compact operators is closed, we see that $M_{\phi_i} \maD M_{\phi_i}$ is compact.

Therefore, it follows that $\oid + \maD$ is Fredholm if, and only if, 
$\oid + p \maD p$ is Fredholm.
The operator $\oid + p \maD p$ is diagonal with respect to the orthogonal decomposition 
$L^2(M,E) =L^2(M\smallsetminus \maP , E\vert_{M\smallsetminus \maP}) \oplus L^2(\maP,E\vert_\maP)$.
Thus, $\oid + \maD$ is Fredholm if, and only if, 
$\oid_{L^2(M\smallsetminus \maP,E\vert_{M\smallsetminus \maP})} + p \maD p$ is Fredholm.

Since for any $\phi, \psi \in \maC(M\smallsetminus \maP)^G$ with disjoint support, 
\begin{equation*}
M_\phi\big(\oid_{L^2(M\smallsetminus \maP,E\vert_{M\smallsetminus \maP})} + p \maD p\big) M_\psi=M_\phi \maD M_\psi 
\end{equation*}
is a compact operator, we get that
\begin{equation*}
\oid_{L^2(M\smallsetminus \maP,E\vert_{M\smallsetminus \maP})} + p \maD p \in \Psi_{M\smallsetminus \maP/G}\big(L^2(M\smallsetminus \maP,E\vert_{M\smallsetminus \maP})\big).
\end{equation*}  
Consequently, the result follows from Proposition \ref{prop.gen.Simonenko}
because on $M\smallsetminus \maP$ the property of strong convergence to $0$ is satisfied. 
\end{proof}

\begin{remark}
In the previous proof we could have shown that 
$\oid_{L^2(\maP,E\vert_\maP)} + M_{1-\chi} \maD M_{1-\chi}$ is Fredholm using
the pseudodifferential uniformization principle from \cite{Sternin.unif}
and the fact that $S^*_GM\vert_{\maP}=S^*_G\maP=\emptyset$.
Indeed, in this case $\oid_{L^2(\maP,E\vert_\maP)} + M_{1-\chi} \maD M_{1-\chi}$
is elliptic in the sense of \cite[Definition 5.1]{Sternin.unif} and therefore Fredholm. 
\end{remark}

\begin{corollary}\label{cor.Simonenko.Gop}
Let $P \in \Psi(M,E)$ and $\maD \in \Psi_G(M,E)$.
Then $P + \maD$ is Fredholm if, and only if, $P$ is elliptic on $\maP$ and 
the restriction of $P+\maD$ to $M\smallsetminus \maP$ is locally invertible. 
\end{corollary}

\begin{proof}
For  $P \in \Psi(M,E)$ and $f_1,f_2 \in \maC(M)$ with disjoint supports, the operator $M_{f_1} P M_{f_2}$ is compact. 
Therefore, writing
\begin{equation*}
P + \maD = p(P + \maD)p + (\oid-p)(P + \maD)(\oid-p)+ (\oid-p)(P + \maD)p + p(P + \maD)(\oid-p),
\end{equation*}
we see that $(\oid-p)(P + \maD)p$ and $p(P + \maD)(\oid-p)$ are compact.
Moreover, $(\oid-p)(P + \maD)(\oid-p)=(\oid-p)P(\oid-p) + (\oid-p)\maD(\oid-p)$ and
in the last proof we showed that $(\oid-p)\maD(\oid-p)$ is compact.
Thus, $(\oid-p)(P + \maD)(\oid-p) \in \maL(L^2(\maP,E\vert_\maP))$ is Fredholm if, and only if, 
$(\oid-p)P (\oid-p) \in \Psi(\maP,E\vert_\maP)$ is Fredholm.
Since a classical pseudodifferential operator is Fredholm if, and only if, it is elliptic, 
it follows that $(\oid-p)(P + \maD)(\oid-p) \in \maL(L^2(\maP,E\vert_\maP))$ is Fredholm if, and only if, 
$(\oid-p)P (\oid-p) \in \Psi(\maP,E\vert_\maP)$ is elliptic.

Now, $p(P + \maD)p \in \Psi_{(M\smallsetminus \maP)/G}(L^2(M\smallsetminus \maP,E))$ and 
on $M\smallsetminus \maP$ the property of strong convergence to $0$ is satisfied. 
Thus, we can apply Proposition \ref{prop.gen.Simonenko} 
to the operator $p(P + \maD)p$ and we obtain that 
$p(P + \maD)p : L^2(M\smallsetminus \maP ,E) \rightarrow L^2(M\smallsetminus \maP ,E)$ 
is Fredholm if, and only if, it is locally invertible on $(M \smallsetminus \maP) /G$.
Using again that $p(P + \maD)p + (\oid - p) (P + \maD)(\oid-p)$ 
is diagonal with respect to the orthogonal decomposition 
$L^2(M,E) =L^2(M\smallsetminus \maP , E\vert_{M\smallsetminus \maP}) \oplus L^2(\maP,E\vert_\maP)$,
we obtain  the result.
\end{proof}

\section{The Fredholm property for finite groups}\label{sec.finite}
In this section, we shall apply the main result of \cite{BCLN2}, see also \cite{BCLN1,BCN}, 
to characterize Fredholm $G$-operators when $G$ is a finite group.
To avoid any confusion, we shall denote our group by $\Gamma$ in this section.

\subsection*{Pseudodifferential uniformization for finite groups}

Let   $\Gamma$ be a finite group and let $M$ be a compact manifold without boundary.
Let $E \rightarrow M$ be a Hermitian $\Gamma$-vector bundle as before.
Assume that $\Gamma $ acts on $M$ and $E$ by isometries.
%

Denote by $\C[\Gamma]$ the finite dimensional complex vector space of functions over $\Gamma$, 
that is the set of finite sums of the form $\sum_\Gamma a_\gamma \delta_\gamma$,
where $a_\gamma \in \C$ and $\delta_\gamma(\gamma')=1$ if $\gamma=\gamma'$ and $0$ otherwise. 
Notice that $L^2(M \times \Gamma,E)=L^2(M,E\otimes \C[\Gamma])$.

We consider two representations of $\Gamma$:
$$
 \begin{array}{ll}
 1 \otimes R_\gamma : L^2(M\times \Gamma ,E) \rightarrow L^2(M\times  \Gamma ,E), & (1\otimes R_\gamma )(s) (x,\gamma')=s(x,\gamma'\gamma),\vspace{2mm}\\
 T_\gamma \otimes L_\gamma : L^2(M\times \Gamma ,E) \rightarrow L^2(M\times  \Gamma ,E) ,& (T_\gamma\otimes L_\gamma) (s) (x,\gamma')=\gamma s(\gamma^{-1}x,\gamma^{-1}\gamma').
\end{array} 
$$
Clearly, these  representations commute: $( 1 \otimes R_\gamma )( T_{\gamma'} \otimes L_{\gamma'} ) = ( T_{\gamma'} \otimes L_{\gamma'} ) ( 1 \otimes R_\gamma )$ for all $\gamma,\gamma'\in\Gamma$.

We shall need the unitary isomorphism
\begin{equation}
Q : L^2 (M,E \otimes \C[\Gamma]) \rightarrow L^2(M,E\otimes \C[\Gamma]),\qquad Qs(x,\gamma)=\gamma s(\gamma^{-1}x,\gamma^{-1}).
\end{equation}
Notice that $Q^2=\oid$ and therefore $Q^*=Q^{-1}=Q$.

\begin{lemma} \label{lem.iso.commut}
The isomorphism $Q$ intertwines the representations $1 \otimes R$ and $T\otimes L$: $Q(1 \otimes R_\gamma)=(T_\gamma \otimes L_\gamma)Q$ forall $\gamma\in\Gamma$. Moreover, we have the following unitary isomorphism 
\begin{equation}
L^2(M,E) \cong L^2(M\times \Gamma,E)^{1\otimes R} \cong L^2(M \times \Gamma,E)^{T\otimes L}.
\end{equation}
\end{lemma}

\begin{proof}
We have 
\begin{align*}
Q(1\otimes R_\gamma)s(x,\gamma')&=\gamma'(1 \otimes R_\gamma)s(\gamma'^{-1}x,\gamma'^{-1})\\
&=\gamma's(\gamma'^{-1}x,\gamma'^{-1}\gamma), 
\end{align*}
and 
\begin{align*}
(T_\gamma \otimes L_\gamma)Qs(x,\gamma')&=\gamma Qs(\gamma^{-1}x,\gamma^{-1}\gamma')\\
&=\gamma\gamma^{-1}\gamma's(\gamma'^{-1}\gamma \gamma^{-1}x,\gamma'^{-1}\gamma),\\
&=\gamma's(\gamma'^{-1}x,\gamma'^{-1}\gamma). 
\end{align*}
The first isomorphism is given by $s \mapsto s \otimes 1_\Gamma$, i.e. $(s\otimes 1_\Gamma)(x,\gamma) = s(x)$. 
Indeed, 
$$
 (1 \otimes R_\gamma)(s \otimes 1_\Gamma)(x, \gamma')=(s \otimes 1_\Gamma)(x,\gamma'\gamma)=s(x)=(s\otimes 1_\Gamma)(x,\gamma').
 $$
Now, if $s\in L^2(M\times \Gamma , E)^{1\otimes R}$ then $s(x,\gamma)=(1 \otimes R_\gamma s)(x,e)=s(x,e)$.
These two morphisms are clearly isometric (with respect to the normalized Haar measure) and inverse to each other.

The second isomorphism is given by the unitary $Q$ since $Q(1\otimes R_\gamma)=(T_\gamma \otimes L_\gamma)Q$.
\end{proof}

The following theorem was shown in~\cite{Sternin.unif} for general compact Lie groups. Since here we are only interested in finite groups (see also \cite{Antonevich1973}), we give the proof in this simpler case for the benefit of the reader.

\begin{theorem} 
\label{thm.unif}
Let $\maD=\sum_{\gamma\in\Gamma} D_\gamma T_\gamma \in \Psi_\Gamma(M,E)$ and 
\begin{equation}
\hat{\maD}:=\sum_{\gamma\in\Gamma} Q\big(D_\gamma \otimes \oid_{\C[\Gamma]}\big) Q (1 \otimes R_\gamma): L^2(M,E \otimes \C[\Gamma])\longrightarrow L^2(M,E \otimes \C[\Gamma]) .
\end{equation} 
\begin{enumerate}\renewcommand{\labelenumi}{\rm (\alph{enumi})}
\item Then $\hat{\maD}  \in \Psi(M,E \otimes \C[\Gamma])$ 
is a 
$\Gamma$-invariant classical pseudodifferential operator on $M$ with respect to the $(T\otimes L)$-action with 
coefficients in $E\otimes \C[\Gamma] \rightarrow M$, i.e. $\hat{\maD} \in \Psi(M,E \otimes \C[\Gamma])^{T\otimes L}$.
\item Moreover, the restriction of $\hat{\maD}$ to $L^2(M,E\otimes \C[\Gamma])^{T\otimes L}$ is isomorphic to $\maD$.
\end{enumerate}
\end{theorem}

\begin{proof}
(a)  First, $\hat{\maD}$ is clearly $\Gamma$-invariant with respect to the $(T \otimes L)$-action. 
Indeed, we have 
$( 1 \otimes R_\gamma )( T_{\gamma'} \otimes L_{\gamma'} ) = ( T_{\gamma'} \otimes L_{\gamma'} ) ( 1 \otimes R_\gamma )$. That 
$Q( T_{\gamma'} \otimes L_{\gamma'} ) = (1 \otimes R_{\gamma'}) Q$ follows from  Lemma \ref{lem.iso.commut}  and $Q^2=\oid$. Finally, 
$(D_\gamma \otimes \oid_{\C[\Gamma]})(1 \otimes R_{\gamma'}) = (1 \otimes R_{\gamma'})(D_\gamma \otimes \oid_{\C[\Gamma]})$.
This implies that $\hat{\maD}(T_{\gamma'}\otimes L_{\gamma'}) =(T_{\gamma'}\otimes L_{\gamma'})\hat{\maD}$. 

Moreover,  the fact that				
\begin{align*}
Q(s \otimes \delta_{\gamma'})(x,\gamma) & = T_\gamma(s)(x) \delta_{\gamma'}(\gamma^{-1})= T_{\gamma'^{-1}}(s)(x)\delta_{\gamma'^{-1}}(\gamma)
\end{align*}
implies that 
\begin{equation}\label{QDIdQ}
Q\big(D_\gamma \otimes \oid_{\C[\Gamma]}\big) Q (s \otimes \delta_{\gamma'})
=(T_{\gamma'} D_\gamma T_{\gamma'}^{-1} s) \otimes \delta_{\gamma'}.
\end{equation}
Therefore, $\hat{\maD}$ is a $(T\otimes L)$-invariant element of $\Psi(M,E \otimes \C[\Gamma]))$.  

(b) Since the unitary $Q$ intertwines the representations $T\otimes L$ and $1 \otimes R$, we have that
$Q\hat{\maD}Q$ is the $(1 \otimes R)$-invariant operator
\begin{equation*}
Q\hat{\maD}Q =  \sum_\Gamma (D_\gamma \otimes \oid_{\C[\Gamma]})(T_\gamma \otimes L_\gamma).
\end{equation*} 
Recall that the first isomorphism of Lemma \ref{lem.iso.commut} is given for 
$s \in L^2(M,E)$ and $f\in L^2(M, E \otimes \C[\Gamma])^{T \otimes L}$ by 
\begin{equation*}
Vs= s \otimes 1_\Gamma \qquad  \text{and} \qquad V^{-1}(f)(x)=f(x,e).
\end{equation*}
Since $(T_\gamma \otimes L_\gamma) V(s)=V(T_\gamma s)$, 
it follows that $V$ intertwines the restrictions of $Q \hat{\maD}Q$ and 
$\maD = \sum_\Gamma D_\gamma T_\gamma$. This completes the proof.
\end{proof}

\subsection*{Characterization of Fredholm $\Gamma$-operators}

Let as before $\Gamma$ be a finite group acting by isometries on a compact manifold without boundary $M$  and on a Hermitian $\Gamma$-vector bundle $F \rightarrow M$.
Let $\Gamma_x:=\{\gamma \in \Gamma , \gamma x=x\}$ be the stabiliser of $x$ in $\Gamma$.

\subsubsection*{The case when $M/\Gamma$ is connected} 

We assume that $M/\Gamma$ is connected. 
Then the principal orbit theorem, see \cite{Bredon,tomDieckTransBook}, implies that there is a subgroup  $\Gamma_0 \subset \Gamma$ such that 
\begin{itemize}
\item the set $M_{(\Gamma_0)}:=\{x \in M,\ \Gamma_x \ \mbox{is conjugated with } \Gamma_0\}$ is an open dense submanifold of $M$;
\item $\forall x\in M$, $\Gamma_x$ contains a subgroup conjugated with $\Gamma_0$.
\end{itemize}

Such a subgroup $\Gamma_0$ is called a \emph{minimal isotropy subgroup} 
and $M_{(\Gamma_0)}$ is called the \emph{principal orbit bundle}.

The following result was obtained in~\cite[Theorem~1.5 and Proposition~5.9]{BCLN2}, see also \cite{BCLN1,BCN}.

\begin{theorem}
\label{thm.BCLN}
Let 
$P \in \psi^m(M;F)^\Gamma$, for some $m
\in \R$.  The following are equivalent:
\begin{enumerate}\renewcommand\labelenumi{\rm (\roman{enumi})}
\item The principal symbol $\sigma_m(P)$ defines by restriction to $\Gamma_0$-invariant vectors an
  isomorphism
      \begin{equation*}\label{mainsmbl1}
        \sigma_m(P)(\xi):F_\xi^{\Gamma_0}\longrightarrow F_\xi^{\Gamma_0} \qquad \text{for all }\xi\in (T^*M^{\Gamma_0} \smallsetminus \{ 0\}) 
      \end{equation*}
 where $\Gamma_0$ is the minimal isotropy subgroup introduced above.     
\item The restriction   $P^\Gamma : H^{s}(M,F)^\Gamma \longrightarrow H^{s-m}(M,F)^\Gamma$
 of $P$ to the subspaces of $\Gamma$-invariant sections is Fredholm.      
\end{enumerate}  
\end{theorem}


Notice that $\Psi(M,E) \subset \Psi_\Gamma(M,E)$ since $\Gamma$ is discrete.
Therefore, in particular, $\Psi_\Gamma(M,E)$ is unital, i.e. $\oid \in \Psi_\Gamma(M,E)$.
We can now give the main result of this section.

\begin{theorem}\label{thm.main}
Let $\maD \in \Psi_\Gamma(M,E)$ be a $\Gamma$-pseudodifferential operator. 
Then the following assertions are equivalent.
\begin{enumerate}\renewcommand{\labelenumi}{\rm(\roman{enumi})}
\item The operator $\maD:L^2(M,E)\to L^2(M,E)$ is Fredholm.
\item The restriction to $L^2(M,E \otimes \C[\Gamma])^{T\otimes L}$ 
of the operator $\hat{\maD}$ introduced in Theorem \ref{thm.unif} is Fredholm.
\item The restriction 
\begin{equation}\label{eq-smbl31}
 \sigma_0(\hat{\maD})(\xi) : \big(E_\xi \otimes \C[\Gamma]\big)^{\Gamma_0} \rightarrow \big(E_\xi \otimes \C[\Gamma]\big)^{\Gamma_0}  
\end{equation}
is invertible $\forall \xi \in S^*M^{\Gamma_0}$.
\end{enumerate} 
\end{theorem}

\begin{proof}
The  equivalence $(\rm{i}) \Leftrightarrow (ii)$ is a consequence of Theorem~\ref{thm.unif},
while the equivalence $(\rm{ii}) \Leftrightarrow (iii) $ 
follows from Theorem~\ref{thm.BCLN} applied with $P=\hat{\maD}$ and  $F=E \otimes \C[\Gamma]$. 
\end{proof}

\subsubsection*{The case when $M/\Gamma$ is not connected} 

We now explain how to reduce the general case to the case, where  $M/\Gamma$ is connected.

Let $\pi : M \to M/\Gamma$ be the quotient map and consider the decomposition
$M/\Gamma = \sqcup_{i=1}^N C_i$ as the {\em disjoint} union
of its connected components. 
We then introduce the preimages $M_i := \pi^{-1}(C_i)$
 of these connected components. In
general, the submanifolds $M_i$ are not connected, but, for each $i$,
$M_i/\Gamma = C_i$ {\em is connected} and   $M_i$ is $\Gamma$-invariant. 


We shall denote by $\chi_i \in C^\infty(M)^\Gamma$ the characteristic function of $M_i$ and
by $E_i := E\vert_{M_i}$ the restriction of a $\Gamma$-equivariant vector bundle $E$ over $M$ to $M_i$. 
Notice that $p_i:=M_{\chi_i} : L^2(M,E) \rightarrow L^2(M,E)$ is a $\Gamma$-invariant projection
such that $p_ip_j=0$ for all $i\neq j$ and $\oid = \sum p_i$.
Therefore, we have 
\begin{equation}\label{eq.dir.sum}
  L^2(M; E) \, \simeq \, \oplus_{i=1}^N L^2(M_i; E_i) \quad \mbox{ and
  } \quad
  \oplus_{i=1}^N \Psi_\Gamma(M_i; E_i) \, \subset \, \Psi_\Gamma(M,E)\,.
\end{equation}

Recall that $\maK(\maH)$ denotes the algebra of compact operators on a
Hilbert space $\maH$. The following proposition provides the desired
reduction to the connected case.

\begin{proposition}\label{prop.reduction}
Let $p_i:= M_{\chi_i} : L^2(M; E) \to L^2(M_i; E_i)$ be the canonical orthogonal
projection as before. For $\maD \in \Psi_\Gamma(M,E)$, we have
\begin{equation*}
\maD = \sum_{i} \maD_i + K,\quad\text{where $\maD_i = \sum p_i D_\gamma p_i T_\gamma \in \Psi_\Gamma(M_i,E_i)$ }
\end{equation*}
and $K \in \maK(L^2(M,E))$ is compact.
Moreover, $\maD$ is Fredholm if, and only if, all $\maD_i$ are Fredholm.
In particular, $\maD$ is Fredholm if, and only if, for every $i$, the operator $\hat{\maD}_i$ satisfies the condition in Theorem \ref{thm.main}(iii).
\end{proposition}

\begin{proof}
We write
\begin{equation*}
\maD = \sum_{i,j} p_i \maD p_j =\sum_i p_i \maD p_i + K ,
\end{equation*}
where $K=\sum_{i\neq j} p_i \maD p_j$. 
By $\Gamma$-invariance of $\chi_i$, 
we obtain $T_\gamma p_i =T_\gamma M_{\chi_i} = M_{\chi_i} T_\gamma = p_i T_\gamma$.
It follows that
\begin{equation*}
p_i \maD p_j = \sum_\Gamma p_iD_\gamma p_j T_\gamma.
\end{equation*} 
Since $p_i D_\gamma p_i$ is a classical 
pseudodifferential operator on $M_i$ with coefficients in $E_i$,
we get that $\maD_i=p_i\maD p_i \in \Psi_\Gamma(M_i,E_i)$.
Moreover, if $i \neq j$ then $p_i D_\gamma p_j$ has zero principal symbol and hence is compact.
It follows that $K=\sum_{i\neq j} p_i \maD p_j$ is a finite sum of compact operators
and therefore is compact. 

The rest follows from Equation \eqref{eq.dir.sum} 
since $\sum_i \maD_i$ is a diagonal operator with respect to the direct sum 
$L^2(M,E) \cong \bigoplus_i L^2(M_i,E_i)$. 
\end{proof}

\subsection*{Application}
We now give the following result regarding the index of a Fredholm $\Gamma$-operator, 
see also \cite[Theorem 5.3]{BCLN2} and \cite{Nistor.Higher.Index} for more details. 
We denote by $\Sigma$ the $C^*$-algebra of symbols $ {\sigma}_0(\hat{\maD})$ with $\maD \in \Psi_\Gamma(M,E)$ of order zero
and $ {\sigma}_0(\hat{\maD})$ as in Theorem~\ref{thm.main}. In other words,
$\Sigma$ is the range of the restriction homomorphism 
$$
\maC(S^*M,\End(E \otimes \C[\Gamma]))^\Gamma \longrightarrow \maC(S^*M^{\Gamma_0},\End((E\otimes \C[\Gamma])^{\Gamma_0})).
$$
 We shall denote by $\hat{\Psi}(M,E)$ the $C^*$-algebra 
of restrictions to $L^2(M,E\otimes \C[\Gamma])^{T\otimes L}$ of elements of $\Psi(M,E\otimes \C[\Gamma])^{T\otimes L}$, see Theorem \ref{thm.unif}. Notice that $\hat{\Psi}(M,E) \cong \Psi(M,E)$.

\begin{proposition}
Let $\maD \in \Psi_\Gamma(M,E)$ be Fredholm. Then the Fredholm index of $\maD$ is given by
$$\Ind(\maD)=\partial [\widehat{\sigma}_0(\hat{\maD})],$$
where 
$\partial : K_1(\Sigma) \rightarrow K_0(\maK(L^2(M,E)))\cong \Z$ 
is the boundary map associated with the exact sequence
\begin{equation}\label{eq.sequence}
0  \longrightarrow \maK(L^2(M,E)) \longrightarrow \hat{\Psi}_\Gamma(M,E) \longrightarrow \Sigma  \longrightarrow 0.
\end{equation}
\end{proposition}  

\begin{proof}
The fact that the sequence in~\eqref{eq.sequence} is exact follows from the isomorphism 
$L^2(M,E\otimes \C[\Gamma])^{T\otimes L}\cong L^2(M,E)$ and Proposition \ref{prop.reduction}.
The proof is completed using that the index morphism is the boundary map in $K$-theory for the Calkin exact sequence, e.g., see~\cite[Theorem 5.3]{BCLN2} and~\cite{Nistor.Higher.Index}.
\end{proof}

\begin{remark}
We refer to \cite{Antonevich1973,NSSBook,Savin.unif} for index theorems under the assumption that $\sigma_0(\hat{\maD}) \in \maC(S^*M, \End(E\otimes \C[\Gamma]))^\Gamma$ is invertible.
\end{remark}

\subsection*{Particular cases}

In this section, we discuss particular cases when $E=M\times \C $ with the trivial action of $\Gamma$ on $\C$. We shall denote simply $\Psi_\Gamma(M,E)=\Psi_\Gamma(M,\C )$ by $\Psi_\Gamma(M)$  and similarly $\Psi(M,E)$ by $\Psi(M)$.   Finally, we assume that $M/\Gamma$ is connected. 

\medskip

$\bullet$ If $\Gamma$ acts trivially on $M$ then $\Psi_\Gamma(M)=\Psi(M)$
and $\maD \in \Psi_\Gamma(M)$ is Fredholm if, and only if, $\maD$ is elliptic.
This is consistent with Theorem \ref{thm.main} because by Equation \eqref{QDIdQ} and  
Egorov's theorem we have  
\begin{align*}
\sigma_0\Big(\sum Q(D_\gamma \otimes \oid_{\C[\Gamma]})Q (1 \otimes R_\gamma)\Big)
&=\sigma_0\bigg(\sum_\gamma \underset{\gamma'}{\operatorname{diag}}\big(T_{\gamma'} D_\gamma T_{\gamma'^{-1}}\big)  (1 \otimes R_\gamma)\bigg)\\
&=\sum_\gamma \underset{\gamma'}{\operatorname{diag}}\big( \gamma'^{*-1}(\sigma_0(D_\gamma))\big) R_\gamma\\
&=\sum_\gamma \underset{\gamma'}{\operatorname{diag}}\big( \sigma_0(D_\gamma)\big) R_\gamma,
\end{align*}
where $\underset{\gamma'}{\operatorname{diag}}(A_{\gamma'})(v \otimes \delta_{\gamma''})=(A_{\gamma''} v) \otimes \delta_{\gamma''}$
and in the last equality we have used that the action is trivial on $M$. 
By triviality of the action on $M$, we have $\Gamma_0=\Gamma$ and therefore 
$(S^*M \times \C[\Gamma])^{\Gamma_0}=S^*M \times \C[\Gamma]^\Gamma=S^*M \times \C$.
Moreover, $(1 \otimes R_\gamma) \vert_{ (S^*M \times \C[\Gamma])^{\Gamma_0}}= \oid_{\C[\Gamma]^\Gamma}$ and thus
$$\sigma_0\Big(\sum Q(D_\gamma \otimes \oid_{\C[\Gamma]})Q (1 \otimes R_\gamma)\Big)\vert_{ (S^*M \times \C[\Gamma])^{\Gamma_0}}
= \underset{\gamma'}{\operatorname{diag}}\Big( \sum \sigma_0(D_\gamma)\Big)\vert_{(S^*M \times \C[\Gamma])^{\Gamma_0}} $$
is invertible if, and only if, $ \sum \sigma_0(D_\gamma)$ is invertible.
In other words, $\maD$ is Fredholm if, and only if, $\maD$ is elliptic. 

\medskip

$\bullet$ Recall that the action of a discrete  group is {\em topologically free}
if for all $N$ and $\gamma_1,...,\gamma_N\in \Gamma\setminus\{e\}$ the union   $\cup_{j=1}^N M^{\gamma_j}$ of fixed point sets does not contain an open set in $M$. If  $\Gamma$ is a finite group, then this condition is equivalent to the condition the, for any open set $U \subset M$, there is $x \in U$ such that all $\gamma x$ are distinct. Note that this condition is equivalent to the condition that $\dim M^\gamma<\dim M$ for all $\gamma\in\Gamma\setminus\{e\}$. It is also equivalent to the condition that $\Gamma_0=\{e\}$, where   $\Gamma_0\subset\Gamma$ is a minimal isotropy subgroup of the action. 
Assume that the action of $\Gamma$ is topologically free on $M$.
In this case, it was shown in \cite{AL94,
ABL98b}
that $\maD \in \Psi_\Gamma(M)$ is Fredholm if, and only if, its trajectory symbol is invertible.
In other words, $\maD \in \Psi_\Gamma(M)$ is Fredholm if, and only if, $\hat{\maD}$ is elliptic.
Again this is consistent with Theorem~\ref{thm.main}.


 $\bullet$ If a minimal isotropy subgroup $\Gamma_0\subset\Gamma$ of the action is a nontrivial normal subgroup, then the restriction of the action to $\Gamma_0$ is trivial. Hence, we obtain a topologically free action of the quotient group $\Gamma/\Gamma_0$ and we can write a $\Gamma$-operator as a $\Gamma/\Gamma_0$-operator:
 $$
 \mathcal{D}=\sum_{\gamma\in\Gamma} D_\gamma T_\gamma=\sum_{\langle\gamma\rangle\in \Gamma/\Gamma_0}\left(\sum_{\gamma'\in  \langle\gamma\rangle} D_{\gamma'}\right)T_\gamma.
 $$
 Then one can show that the symbol \eqref{eq-smbl31} for $\mathcal{D}$ as a $\Gamma$-operator is isomorphic to the symbol of this operator as a $\Gamma/\Gamma_0$-operator for the topologically free action of  $\Gamma/\Gamma_0$.
 
  $\bullet$ Let $H\subset \Gamma$ be a subgroup and consider $M=X\times (\Gamma/H)$, where $X$ is a connected closed smooth manifold. Let   $\Gamma$ act by left multiplications: $\gamma\in\Gamma$ takes  $(x,\gamma' H)  $ to $(x, \gamma \gamma' H)$. Then the minimal isotropy subgroup is equal to $\Gamma_0=H$ and we have an isomorphism 
$$
    J:C^\infty(X\times (\Gamma/H))\to C^\infty(X,\mathbb{C}[\Gamma]^H)
$$ 
of $\Gamma$-representations. This isomorphism takes a $\Gamma$-operator 
$$
D=\sum_{\gamma\in \Gamma} D_\gamma T_\gamma
$$
to the matrix pseudodifferential operator
$$
 JDJ^{-1}=\sum_\gamma D_\gamma  L^{-1}_\gamma,
$$
where $L_\gamma: \mathbb{C}[\Gamma]^{H}\to  \mathbb{C}[\Gamma]^{H}$ is the operator of left multiplication by $\gamma$: $(L_\gamma u)(\gamma')=u(\gamma \gamma')$.  Then a direct computation shows that the symbol~\eqref{eq-smbl31} of the $\Gamma$-operator $D$ is isomorphic to the symbol of the matrix pseudodifferential operator $JDJ^{-1}$.

\setlength{\baselineskip}{4.75mm}
\bibliographystyle{plain}


\end{document}